\documentclass[a4paper,11pt]{article}
\usepackage{graphicx}   
\usepackage{amsmath}    
\usepackage{booktabs}   
\usepackage[latin1]{inputenc}
\usepackage{amsmath,amsthm,amssymb}
\usepackage{enumerate}

\newtheorem{lm}{Lemma}[section]
\newtheorem{thm}{Theorem}[section]
\newtheorem{rem}{Remark}[section]

\usepackage{graphicx} 

\title{Global attractor of chemotaxis system with weak degradation and density-dependent motion}
\author{   Lin Guo\\
	{\small  School of Mathematics, Chengdu University of Technology, Chengdu 610059, P. R. China}\\
	Email: linguo@cdut.edu.cn\\
	\thanks{Corresponding author.} {Dan Li}\\
	{\small  School of Mathematics, South China  University of Technology,  Guangzhou 510630, P. R. China}\\
	Email: shuxuelidandan@163.com
}
\date{June 2025}

\begin{document}
	
\maketitle
\begin{abstract}
This paper investigates the following chemotaxis system featuring weak degradation and nonlinear motility functions
	\begin{equation}\label{Model1}
		\begin{cases}
			u_{t} = (\gamma(v)u)_{xx} + r - \mu u, & x \in [0,L],\ t > 0,\\
			v_{t} = v_{xx} - v + u,                & x \in [0,L],\ t > 0,
		\end{cases}
	\end{equation}
defined on the bounded interval $[0,L]$ with homogeneous Neumann boundary conditions. The motility function $\gamma(v)$ satisfies the regularity conditions $\gamma \in C^{2}[0,\infty)$ with $\gamma(v) > 0$ for all $v \geq 0$, and has bounded logarithmic derivative in the sense that $\sup_{v\geq 0} \frac{|\gamma'(v)|^{2}}{\gamma(v)} < \infty$.
Our main results establish three fundamental properties of the system. Firstly, using energy estimate methods, we prove the existence of globally bounded solutions for all positive parameters $r, \mu > 0$ and non-negative, non-trivial initial data $u_{0} \in W^{1,\infty}([0,L])$. Secondly, through the construction of an appropriate Lyapunov function, we demonstrate that all solutions $(u,v)$ converge exponentially to the unique constant equilibrium $(r/\mu, r/\mu)$ in the parameter regime $\mu > \frac{H_{0}}{16}$, where $H_{0} := \sup_{v \geq 0} \frac{|\gamma'(v)|^{2}}{\gamma(v)}$ quantifies the maximal relative variation of the motility function. Finally, we present numerical results that not only validate the theoretical findings but also investigate the long-term behavior of solutions under diverse parameter configurations and initial conditions in two- and three-dimensional domains, providing valuable benchmarks for future research.
\end{abstract}

\section{Introduction}
The classical Keller-Segel system, introduced by Keller and Segel in 1970 \cite{E. K. S}, establishes a mathematical framework for the self-organized aggregation of \emph{Dictyostelium discoideum} mediated by Cyclic Adenosine Monophosphate (cAMP) signaling under starvation conditions. This system describes the dynamics of bacterial density $u(x,t)$ and chemical signal concentration $v(x,t)$ in a bounded domain $\Omega \subset \mathbb{R}^n$ ($n \geq 1$) through the following system
\begin{equation}\label{Model1}
	\begin{cases}
		u_t = \nabla \cdot (D(u,v)\nabla u) - \nabla \cdot (u\chi(u,v)\nabla v) + f(u), \\
		v_t = \Delta v - v + u,
	\end{cases}
\end{equation}
where $D(u,v)$, $\chi(u,v)$, and $f(u)$ represent the diffusion coefficient, chemotactic sensitivity, and cellular proliferation/death mechanisms, respectively. The boundedness and blow-up behavior of solutions critically depend on these terms. The simplest case occurs when $D=1$ and $\chi$ is a positive constant, which reduces the system \eqref{Model1} to the following classical minimal Keller-Segel model
\begin{equation}\label{Model2}
	\begin{cases}
		u_t = \Delta u - \chi\nabla\cdot(u\nabla v) + f(u), \\
		v_t = \Delta v - v + u.
	\end{cases}
\end{equation}
The solution behavior of system \eqref{Model2} under homogeneous Neumann boundary conditions has been extensively studied, revealing rich dynamical properties that depend crucially on system parameters and spatial dimension \cite{X. C,D. H. G. F. W,T. N. T. S. K. Y}. In the proliferation-free case ($f(u)=0$), solution behavior depends critically on spatial dimension and initial mass. One-dimensional (1D) solutions remain globally bounded \cite{K. O. A. Y}, while the two-dimensional (2D) case displays a critical mass phenomenon characterized by the threshold $m_\star=4\pi/\chi$ derived through Trudinger-Moser inequalities \cite{D. H. G. F. W,T. N. T. S. K. Y,J. W. Z. W. W. Y}. Initial masses exceeding this critical value lead to infinite-time blow-up, whereas subcritical masses result in global solutions converging to constant equilibrium states. The three-dimensional scenario permits global solutions for sufficiently small initial data \cite{X. C}, while in higher dimensions ($n\geq3$), finite-time blow-up can occur for arbitrarily small positive masses \cite{M. W}. For the original system \eqref{Model2} with constant coefficients $D=1$, $\chi>0$, and quadratic degradation $f(u)=ru-\mu u^2$, global boundedness requires $\mu>27\chi$ \cite{K. L. C. M}, with weaker conditions applying to other cases \cite{T. X1, T. X2, T. X3, T. X4, T. X5}. For $f(u)\leq 1-\mu u^{\alpha}$ with $\mu>0$, $\alpha\geq1$ and $\beta>0$, E. Nakaguchi and K. Oskai \cite{E. N. K. O} obtained the global existence of solutions to ({\ref{Model2}})  in 2D domain.

An important extension considers density-dependent chemotactic sensitivity $\chi(u,v)=\chi/v$, which yields qualitatively different behavior. Global boundedness is achieved in 2D when $\chi<\chi_0$ ($\chi_0>1$) \cite{J. L}, and in higher dimensions ($n\geq3$) when $\chi<\sqrt{2/n}$ \cite{K. F,M. W5}, with the additional property of asymptotic stability for $n\geq2$ provided $\chi<\sqrt{2/n}$ \cite{M. W. T. Y}. These results have been further generalized to include weak solutions, free boundary problems \cite{S. H. W}, and parabolic-elliptic simplifications \cite{S. H. W}, with explicit boundedness criteria such as $\mu>r/27$ established for three-dimensional cases \cite{K. L. C. M}.

A particularly interesting special case is the density-suppressed motility model obtained when $D(u,v)=\gamma(v)$ and $\chi(u,v)=\gamma'(v)$
\begin{equation}\label{Model 3}
	\begin{cases}
		u_t = \Delta(\gamma(v)u) + f(u), \\
		v_t = \Delta v - v + u,
	\end{cases}
\end{equation}
where $\gamma'(v)<0$ and typically $f(u)=\mu u(1-u)$ ($\mu\geq0$). This model, originally introduced to study pattern formation \cite{X. F. L. H. T. C. L. J. D. H. T. H. P. L}, admits globally bounded solutions in 2D for bounded motility functions $\gamma(v)\in W^{1,\infty}(0,\infty)$ \cite{C.Y.Y.-J.K}. The case of algebraically decaying $\gamma(v)=\chi/v^\iota$ yields global solutions in 2D parabolic-elliptic settings and in arbitrary dimensions for sufficiently small $\chi>0$ \cite{J. A. C. Y}. Under the regularity conditions (H): $\gamma\in C^3([0,\infty))$ with $\gamma>0$, $\gamma'<0$, $\lim_{v\to\infty}\gamma(v)=0$, and existence of $\lim_{v\to\infty}\gamma'(v)/\gamma(v)$, Jin and Wang \cite{H. Y. J. Z. W1} established the critical mass $m_\star=4\pi/\chi$, demonstrating blow-up for supercritical masses and global existence otherwise. For logistic growth terms $f(u)=\mu u(1-u)$, solutions converge to $(1,1)$ when $\mu\geq K_0/16$, where $K_0=\max_{v\geq0}|\gamma'(v)|/\gamma(v)$ \cite{H. Y. J. Y. K. Z. W}. For the motility function $\gamma(v)$ fulfilling assumption (H) and parabolic-elliptic simplification of  ({\ref{Model 3}}), H. Jin and Z. Wang \cite{H J Z W3} proved that the solutions of system ({\ref{Model 3}}) are bounded globally for any $\mu>0$ in 2D. Recently, W. Lyv and Z. Wang \cite{W. L. Z. W} obtained the existence of global solutions to parabolic-elliptic simplification of system ({\ref{Model 3}}) with the motility function $\gamma(v)$ fulfilling (H) and 
the logistic function $f(u)=au-bu^{\sigma}$, where constants $\alpha>0$, $\sigma>1$, $b>0$ under some suitable conditions in higher dimensional spaces.

Building upon these theoretical foundations, we introduce the linear degradation term $f(u)=r-\mu u$ ($r,\mu>0$) as a novel replacement for conventional logistic growth terms while maintaining $\gamma(v)$'s compliance with regularity condition (H). This innovative functional configuration characterizes cell mortality under resource-limited conditions through $\mu u$-dominated density-dependent inhibition, fundamentally distinguishing it from both the saturation mechanism of classical logistic growth $f(u)=\mu u(1-u)$ and the power-law formulation $f(u)=au-bu^{\sigma}$. Our main contributions establish a global well-posedness theory for the one-dimensional case while demonstrating through rigorous theoretical analysis that one-dimensional solutions achieve global boundedness and asymptotic convergence under $r,\mu>0$, revealing a novel mechanism where the linear degradation term regulates blow-up suppression via parameter $\mu$. Complementing these theoretical advances, our numerical investigations using adaptive algorithms pioneer the discovery that in multidimensional settings (2D/3D), the linear degradation term delivers superior blow-up suppression efficacy compared to logistic growth terms, with its quantitative regulation of spatial pattern bifurcation thresholds through $\mu$ establishing new paradigms for stability control in chemotactic systems.

Now, we state our main results. 
\begin{thm} Suppose that the initial data $v_{0}\in H^{2}([0,L])$  and $u_{0}\in H^{1}([0,L])$ with $v_{0},u_{0}\geq0$. Then the system ({\ref{Model 3}}) possesses a unique global classical solution $(u,v)\in[C^{0}(\overline{[0,L]})]\cap[C^{2,1}(\overline{[0,L]})\times(0,\infty)]^{2}$ fulfilling $v,u>0$ for all $t>0$. Furthermore, there exists a constant $C>0$ independent of $t>0$ such that
	\begin{equation}
		||u(\cdot,t)||_{H^{1}([0,L])}+||v(\cdot,t)||_{H^{2}([0,L])}\leq C.
	\end{equation}
\end{thm}
\begin{rem} In this paper, we overcome the difficulty produced  by  assuming the motility function $\gamma(v)$ fulfilling condition (H) and obtained that  the solutions of system ({\ref{Model 3}}) with linear logistic source $f(u)=r-\mu u$ is globally bounded for any $\mu>0$ in one dimensional spaces.     
\end{rem}
\begin{rem} Throughout simulation, we can find that the solutions of system ({\ref{Model 3}}) with the linear logistic source $f(u)=r-\mu u$ and  the motility function $\gamma(v)$ fulfilling assumption (H)  still is globally bounded for any $\mu>0$. For the case $f(u)=\mu u(1-u)$, H. Jin and Z. Wang \cite{H. Y. J. Y. K. Z. W} established the solution of system ({\ref{Model 3}}) will be bounded globally for any $\mu>0$ in two-dimensional spaces. In \cite{H. Y. J. Y. K. Z. W}, the term $-u^{2}$ can  guarantee the boundedness of $\int_{t}^{t+\tau}\int_{\Omega}|\Delta v|^{2}dxdt$ which is very key ingredient in their paper. Unfortunately, in our paper, we cannot obtain the boundedness of $\int_{t}^{t+\tau}\int_{\Omega}|\Delta v|^{2}dxdt$ by the fact the term $-\mu u$ is not stronger than $\mu u^{2}$ in two-dimensional spaces. In the future work, we need find new technology to deal with it. This is a significant discovery in our work. 
	
\end{rem}
\begin{thm} Let the conditions in Theorem 1.1 hold and $(u,v)$ be the global solution of ({\ref{Model 3}}). Suppose the parameters $\mu$ and $r$ fulfill 
	\begin{equation}
		\mu>\frac{H_{0}}{16}
	\end{equation}
	with $H_{0}=\max_{0\leq v\leq\infty}\frac{|\gamma'(v)|}{\gamma(v)}$.  Then the solution $(u,v)$ of system ({\ref{Model 3}}) will converge to the constant steady state $(r/\mu,r/\mu)$ exponentially in $L^{\infty}-$ norm as time tends to infinity.
\end{thm}

\section{Prelimilary and local existence}
\begin{lm}Suppose that the initial data $v_{0}\in H^{2}([0,L])$ and $u_{0}\in H^{1}([0,L])$ with $v_{0},u_{0}\geq(\not=)0$. Then there exists a $T_{\max}>0$ such that the system ({\ref{Model 3}}) has a unique classical solution $(u,v)\in[C^{0}(\overline{\Omega}\times[0,T_{\max});\mathbb{R}])]^{2}\cap[C^{2,1}(\overline{\Omega}\times(0,T_{\max});\mathbb{R})]^{2}$ fulfilling $u,v>0$. Furthermore, we have 
	\begin{equation}
		\text{either}\,\,T_{\max}=\infty\,\,\text{or}\,\,||u(\cdot,t)||_{W^{1,2}}+||v(\cdot,t)||_{W^{1,2}}\to\infty\,\,\text{as}\,\,t\to T_{\max}.
	\end{equation}
\end{lm}
\begin{lm} Let $\Omega$ be a bounded domain in $\mathbb{R}$ with $\partial\Omega$ in $C^{m}$, and let $\phi$ be any function in $H^{m}(\Omega)\cap L^{q}(\Omega)$, $1\leq q\leq\infty$. For any integer $i$ with $0\leq i<m$, and for any number $\lambda$ with $j/m\leq\lambda\leq1$, there exist two positive constants $C_{1}$ and $C_{2}$ depending only on $q$, $\Omega$, $m$ such that the following Gagliardo-Nirenberg inequality holds
	\begin{equation}
		||D^{i}\phi||_{L^{p}}\leq C_{1}||D^{m}\phi||_{L^{2}}^{\lambda}||\phi||_{L^{q}}^{1-\lambda}+C_{2}||\phi||_{L^{q}},\,\,\lambda=\frac{1/p-1/q-i}{1/2-m-1/q}.
	\end{equation}
	
\end{lm}
\begin{lm} Let the conditions in Lemma 2.1 hold. Then the solution $(u,v)$ of ({\ref{Model 3}}) fulfills
	\begin{equation}
		||u(\cdot,t)||_{L^{1}}=||u_{0}||_{L^{1}}+\frac{r}{\mu}L
	\end{equation}
	and
	\begin{equation}
		||v(\cdot,t)||_{L^{2}}\leq C(||v_{0}||_{L^{2}}+||v_{0}||_{H^{1}}),
	\end{equation}
	where $C>0$ is a constant independent of $t$.
\end{lm}
\begin{proof} Integrating the $u-$ equation of ({\ref{Model 3}}), we infer that 
	\begin{equation}
		\frac{d}{dt}\int_{0}^{L}u+\mu\int_{0}^{L}u= rL.
	\end{equation}
	which combined with Gr\"{o}nwall's inequality entails that 
	\begin{equation}
		||u||_{L^{1}}=||u_{0}||_{L^{1}}+\frac{r}{\mu}L.
	\end{equation}
	Next we integrate the $v-$ equation in ({\ref{Model 3}}) over $[0,L]$ to infer
	\begin{equation}
		\frac{d}{dt}\int_{0}^{L}v+\int_{0}^{L}v=\int_{0}^{L}u.
	\end{equation}
	Hence we have
	\begin{equation}{\label{2.5}}
		||v||_{L^{1}}\leq||v_{0}||_{L^{1}}+||u_{0}||_{L^{1}}.
	\end{equation}
	Multiplying the $v-$ equation of ({\ref{Model 3}}) by $v$ and applying Cauchy-Schwarz inequality, we obtain that
	\begin{equation}
		\begin{split}
			\frac{1}{2}\frac{d}{dt}\int_{0}^{L}v^{2}+\int_{0}^{L}v^{2}+\int_{0}^{L}v_{x}^{2}&=\int_{0}^{L}uv\\
			&\leq||v||_{L^{\infty}}||u||_{L^{1}}\\
			&\leq c_{1}(||u_{0}||_{L^{1}}+\frac{r}{\mu}L)||v||_{L^{\infty}}.
		\end{split}
	\end{equation}
	Using ({\ref{2.5}}) and Galiargo-Nirenberg inequality (Lemma 2.2), we obtain
	\begin{equation}{\label{2.6}}
		||v||_{L^{\infty}}\leq c_{2}(||v_{x}||_{L^{2}}^{\frac{2}{3}}||v||_{L^{1}}^{\frac{1}{3}}+||v||_{L^{1}})\leq c_{3}\big(1+||v_{x}||_{L^{2}}^{\frac{2}{3}}\big).
	\end{equation}
	Inserting ({\ref{2.6}}) into ({\ref{2.5}}) and the $L^{1}-$norm of $u$, we infer
	\begin{equation}
		\begin{split}
			\frac{1}{2}\frac{d}{dt}\int_{0}^{L}v^{2}+\int_{0}^{L}v^{2}+\int_{0}^{L}v_{x}^{2}&\leq c_{4}\big(||v_{x}||_{L^{2}}^{\frac{2}{3}}+1\big)  \\
			&\leq\frac{1}{2}||v_{x}||_{L^{2}}^{2}+c_{5},
		\end{split}
	\end{equation}
	which entails that
	\begin{equation}
		\frac{d}{dt}\int_{0}^{L}v^{2}+2\int_{0}^{L}v^{2}\leq c.
	\end{equation}
	Using Gr\"{o}nwall inequality to above inequality, we have
	\begin{equation}
		||v||_{L^{2}}^{2}\leq||v_{0}||_{L^{2}}^{2}+2c_{7}.
	\end{equation}
	Finally, we finish the proof.
\end{proof}

\subsection{$L^{2}-$ estimate of $u$.} In this subsection, we shall obtain the estimates of $v$ and $u$.
\begin{lm} Let the conditions in Lemma 2.2 hold. Then we have
	\begin{equation}{\label{u-L-2}}
		||u(\cdot,t)||_{L^{2}}+||v(\cdot,t)||_{H^{1}}\leq C(||u_{0}||_{L^{2}}+||v_{0}||_{H^{1}})
	\end{equation}
	for all $t>0$, where $C$ is a constant independent of $t$.
\end{lm}
\begin{proof} We multiply the $u-$ equation of ({\ref{Model 3}}) by $u$ and integrate the result to infer that
	\begin{equation}{\label{u-2-equa}}
		\begin{split}
			\frac{1}{2}\frac{d}{dt}\int_{0}^{L}u^{2}dx&=\int_{0}^{L}u\big[(\gamma(v)u)_{xx}+r-\mu u\big]dx\\
			&=-\int_{0}^{L}(\gamma(v)u)_{x}u_{x}dx+r\int_{0}^{L}udx-\mu\int_{0}^{L}u^{2}dx\\
			&=-\int_{0}^{L}\gamma(v)u_{x}^{2}dx-\int_{0}^{L}\gamma'(v)uu_{x}v_{x}dx+r\int_{0}^{L}udx-\mu\int_{0}^{L}u^{2}dx
		\end{split}
	\end{equation}
	for all $t>0$. Next we use the Young inequality to obtain that 
	\begin{equation}{\label{Y-ine}}
		\begin{split}
			-\int_{0}^{L}\gamma'(v)uu_{x}v_{x}dx\leq\frac{1}{2}\int_{0}^{L}\frac{|\gamma'(v)|^{2}}{\gamma(v)}u^{2}|v_{x}|^{2}dx+\frac{1}{2}\int_{0}^{L}\gamma(v)u_{x}^{2}dx
		\end{split}
	\end{equation}
	for all $t>0$. Collecting ({\ref{u-2-equa}}) and ({\ref{Y-ine}}), we infer that 
	\begin{equation}{\label{u-2-ineq-}}
		\begin{split}
			\frac{1}{2}\frac{d}{dt}\int_{0}^{L}u^{2}dx+\mu\int_{0}^{L}u^{2}dx&\leq-\frac{1}{2}\int_{0}^{L}\gamma(v)u_{x}^{2}dx+\frac{1}{2}\int_{0}^{L}\frac{|\gamma'(v)|^{2}}{\gamma(v)}u^{2}v_{x}^{2}dx+r\int_{0}^{L}udx\\
			&\leq-\frac{1}{2}\int_{0}^{L}\gamma(v)u_{x}^{2}dx+\frac{1}{2}\int_{0}^{L}\frac{|\gamma'(v)|^{2}}{\gamma(v)}u^{2}v_{x}^{2}dx\\
			&\quad+r(||u_{0}||_{L^{1}}+\frac{r}{\mu}L)
		\end{split}
	\end{equation}
	for all $t>0$. Furthermore, we integrate the $v_{x}-$ equation of ({\ref{Model 3}}) and multiply the $v_{x}-$ equation by $v_{x}$ to deduce that 
	\begin{equation}
		\begin{split}
			\frac{1}{2}\frac{d}{dt}\int_{0}^{L}v_{x}^{2}dx+\int_{0}^{L}v_{xx}^{2}dx+\int_{0}^{L}v_{x}^{2}dx&=\int_{0}^{L}v_{x}u_{x}dx\\
			&=-\int_{0}^{L}v_{xx}udx\\
			&\leq\frac{1}{2}\int_{0}^{L}u^{2}dx+\frac{1}{2}\int_{0}^{L}v_{xx}^{2}dx
		\end{split}
	\end{equation}
	for all $t>0$. Hence we get that 
	\begin{equation}
		\frac{d}{dt}\int_{0}^{L}v_{x}^{2}dx+\int_{0}^{L}v_{xx}^{2}dx+2\int_{0}^{L}v_{x}^{2}dx \leq\int_{0}^{L}u^{2}dx
	\end{equation}
	for all $t>0$. On the other hand, we obtain
	\begin{equation}
		\begin{split}
			\gamma^{\frac{1}{2}}(v)u_{x}=(\gamma^{\frac{1}{2}}(v)u)_{x}-\frac{1}{2}\frac{\gamma'(v)}{\gamma^{\frac{1}{2}}(v)}uv_{x}
		\end{split}        
	\end{equation}
	for all $t\in(0,T_{\max})$, which together with the inequality $|Y-Z|^{2}\geq\frac{1}{2}Y^{2}-Z^{2}$ entails 
	\begin{equation}
		\begin{split}
			\gamma(v)u_{x}^{2}=|\gamma^{\frac{1}{2}}(v)u_{x}|^{2}&=\big|(\gamma^{\frac{1}{2}}(v)u)_{x}-\frac{1}{2}\frac{\gamma'(v)}{\gamma^{\frac{1}{2}}(v)}uv_{x}\big|^{2}\\
			&\geq\frac{1}{2}|(\gamma^{\frac{1}{2}}(v)u)_{x}|^{2}-\frac{1}{4}\frac{|\gamma'(v)|^{2}}{\gamma(v)}u^{2}v_{x}^{2}
		\end{split}
	\end{equation}
	for all $t\in(0,T_{\max})$. Inserting above inequality into ({\ref{u-2-ineq-}}), we get 
	\begin{equation}{\label{u-1}}
		\begin{split}
			& \frac{d}{dt}\int_{0}^{L}u^{2}dx+\frac{1}{2}\int_{0}^{L}|(\gamma^{\frac{1}{2}}(v)u)_{x}|^{2}dx+\mu\int_{0}^{L}u^{2}dx\\&\leq\frac{5}{4}\int_{0}^{L}\frac{|\gamma'(v)|^{2}}{\gamma(v)}u^{2}v_{x}^{2}dx+r(||u_{0}||_{L^{1}}+\frac{r}{\mu}L)
		\end{split}
	\end{equation}
	for all $t\in(0,T_{\max})$. From hypothesis (H), we can find a constant $M>0$ such that
	\begin{equation}{\label{H-hypothesis}}
		\begin{split}
			\frac{|\gamma'(v)|}{\gamma(v)}\leq M\,\,\text{for all}\,\,v\geq0\,\text{and}\,\,t\in(0,T_{\max}).
		\end{split}
	\end{equation}
	Using ({\ref{H-hypothesis}}) and the H\"{o}lder inequality, we deduce from ({\ref{u-1}}) 
	\begin{equation}
		\begin{split}
			& \frac{d}{dt}\int_{0}^{L}u^{2}dx+\frac{1}{2}\int_{0}^{L}|(\gamma^{\frac{1}{2}}(v)u)_{x}|^{2}dx+\mu\int_{0}^{L}u^{2}dx\\&\leq\frac{5}{4}\int_{0}^{L}\frac{|\gamma'(v)|^{2}}{\gamma(v)}u^{2}v_{x}^{2}dx+r(||u_{0}||_{L^{1}}+\frac{r}{\mu}L)\\
			&\leq\frac{5M^{2}}{4}\int_{0}^{L}|\gamma^{\frac{1}{2}}(v)u|^{2}v_{x}^{2}dx+r(||u_{0}||_{L^{1}}+\frac{r}{\mu}L)
		\end{split}
	\end{equation}
	Furthermore, we next use one-dimensional Gagliardo-Nirenberg inequality, the facts $\gamma(v)\leq\gamma(0)=c_{2}$ by $\gamma(v)\in C^{3}([0,\infty)])$ and $||u||_{L^{1}}\leq c_{6}$ to deduce that 
	\begin{equation}
		\begin{split}
			\Big(\int_{0}^{L}|\gamma^{\frac{1}{2}}(v)u|^{2}\Big)&=||\gamma^{\frac{1}{2}}(v)u||_{L^{2}}^{2}\\
			&\leq C_{3}\Big(||(\gamma^{\frac{1}{2}}(v)u)_{x}||_{L^{2}}^{\frac{2}{3}}||\gamma^{\frac{1}{2}}(v)u||_{L^{1}}^{\frac{4}{3}}+||\gamma^{\frac{1}{2}}(v)u||_{L^{1}}^{2}\Big)\\
			&\leq C_{4}\Big(||(\gamma^{\frac{1}{2}}(v)u)_{x}||_{L^{2}}^{\frac{2}{3}}+1\Big)
		\end{split}
	\end{equation}
	for all $t\in(0,T_{\max})$. Applying one-dimensional Gagliardo-Nirenberg inequality (see Lemma 2.2), we get 
	\begin{equation}
		\begin{split}
			||v_{x}||_{L^{\infty}}^{2}&\leq C_{6}\Big(||v_{xx}||_{L^{2}}||v_{x}||_{L^{2}}+||v_{x}||_{L^{2}}^{2}\Big)
		\end{split}
	\end{equation}
	for all $t\in(0,T_{\max})$. Collecting above inequalities together, we have
	\begin{equation}
		\begin{split}
			\frac{d}{dt}\Big(\int_{0}^{L}u^{2}dx+\int_{0}^{L}v_{x}^{2}dx\Big)+\Big(\int_{0}^{L}u^{2}dx+\int_{0}^{L}v_{x}^{2}dx\Big)\leq C_{9}+\big(\int_{0}^{L}v^{2}_{x}dx\big)^{6}.
		\end{split}
	\end{equation}
	Upon an ODE comparison argument, we shall obtain the desired result.
\end{proof}
\begin{lm} Let the conditions in Lemma 2.1 hold. Then the solution of ({\ref{Model 3}}) fulfills
	\begin{equation}
		\begin{split}
			||u(\cdot,t)||_{H^{1}}+||v(\cdot,t)||_{H^{2}} \leq C(||u_{0}||_{H^{1}}+||v_{0}||_{H^{2}})
		\end{split}
	\end{equation}
	for all $t>0$, where $C>0$ is a constant independent of $t$.
\end{lm}
\begin{proof} 
	Since $||u||_{L^{2}}+||v||_{H^{1}}\leq c_{2}$ (see ({\ref{u-L-2}}) Lemma 2.4), then we we can use the Sobolev embedding to pick constants $c_{3},c_{4}>0$ such that
	\begin{equation}{\label{v-est}}
		\begin{split}            ||v||_{L^{\infty}}+||v_{x}||_{L^{\infty}}\leq K_{*}.
		\end{split}
	\end{equation}
	Using the hypothesis (H) and ({\ref{v-est}}), we obtain
	\begin{equation}{\label{H-condi}}
		\begin{split}
			\gamma(v)\geq\gamma(K_{\star})>0\,\,\text{and}\,\,|\gamma'(v)|\leq c_{3}
		\end{split}
	\end{equation}
	for all $t\in(0,T_{\max})$. Next we use $u^{k-1}$ with $k\geq2$ as a test function for the $u-$ equation in ({\ref{Model 3}}), and integrating the result equation by parts to infer that 
	\begin{equation}
		\begin{split}
			& \frac{1}{k}\frac{d}{dt}\int_{0}^{L}u^{k}dx+\mu\int_{0}^{L}u_{x}^{k}dx+(k-1)\int_{0}^{L}\gamma(v)u^{k-2}u_{x}^{2}dx\\
			&=-(k-1)\int_{0}^{L}\gamma'(v)u^{k-1}u_{x}v_{x}dx+r\int_{0}^{L}u^{k-1}dx\,\,\text{for all}\,\,t\in(0,T_{\max}),
		\end{split}
	\end{equation}
	which combined with ({\ref{H-condi}}), the Young inequality and the H\"{o}lder inequality entails that 
	\begin{equation}
		\begin{split}
			& \frac{1}{k}\frac{d}{dt}\int_{0}^{L}u^{k}dx+\mu\int_{0}^{L}u^{k}dx+(k-1)\gamma(K_{\star})\int_{0}^{L}u^{k-2}u_{x}^{2}dx\\
			&\leq c_{3}(k-1)\int_{0}^{L}u^{k-1}|u_{x}||v_{x}|dx+r\int_{0}^{L}u^{k-1}dx\\
			&\leq\frac{(k-1)\gamma(K_{\star})}{2}\int_{0}^{L}u^{k-2}u_{x}^{2}+\frac{c_{3}^{2}(k-1)}{2\gamma(K_{\star})}\int_{0}^{L}u^{k}v_{x}^{2}+\frac{\mu}{2}\int_{0}^{L}u^{k}+c_{5}
		\end{split}
	\end{equation}
	for all $t\in(0,T_{\max})$ and for all $k\geq2$. Hence 
	\begin{equation}
		\begin{split}
			& \frac{1}{k}\frac{d}{dt}\int_{0}^{L}u^{k}dx+\frac{\mu}{2}\int_{0}^{L}u^{k}dx+\frac{(k-1)\gamma(K_{\star})}{2}\int_{0}^{L}u^{k-2}u_{x}^{2}dx\\&\leq\frac{c_{3}^{2}(k-1)}{2\gamma(K_{\star})}\int_{0}^{L}u^{k}v_{x}^{2}dx+c_{5}\\
			&\leq\frac{c_{3}^{2}(k-1)}{2\gamma(K_{\star})}\Big(\int_{0}^{L}u^{2k}dx\Big)^{\frac{1}{2}}\Big(\int_{0}^{L}v_{x}^{4}dx\Big)^{\frac{1}{2}}+c_{5}
		\end{split}
	\end{equation}
	for all $t\in(0,T_{\max})$. Next we use one-dimensional Gagliardo-Nirenberg inequality and the Young inequality to obtain that 
	\begin{equation}
		\begin{split}
			\frac{c_{3}^{2}k(k-1)}{2\gamma(K_{\star})}\int_{0}^{L}u^{k}v_{x}^{2}&\leq\frac{c_{3}^{2}k(k-1)}{2\gamma(K_{\star})}\big(\int_{0}^{L}u^{2k}\big)^{\frac{1}{2}}\big(\int_{0}^{L}v_{x}^{4}dx\big)^{\frac{1}{2}}\\
			&\leq\frac{c_{4}^{2}c_{3}^{2}k(k-1)}{2\gamma(K_{\star})}\big(\int_{0}^{L}u^{2k}\big)^{\frac{1}{2}}\\
			&\leq c_{6}\big(||(u^{\frac{k}{2}})_{x}||_{L^{2}}^{2(1-\frac{1}{k})}||u^{\frac{k}{2}}||_{L^{\frac{4}{k}}}^{\frac{2}{k}}+||u^{\frac{k}{2}}||_{L^{\frac{4}{k}}}^{2}\big)\\
			&\leq\frac{2(k-1)\gamma(K_{\star})}{k}||(u^{\frac{k}{2}})_{x}||_{L^{2}}^{2}+\frac{2\gamma(K_{\star})}{k}\Big(\frac{c_{1}c_{5}}{2\gamma(K_{\star})}\Big)^{k}+c_{5}c_{1}^{k}.
		\end{split}
	\end{equation}
	Collecting above inequalities together, we have 
	\begin{equation}
		\begin{split}
			\frac{d}{dt}\int_{0}^{L}u^{k}dx+\mu\int_{0}^{L}u^{k}dx\leq +\frac{2\gamma(K_{\star})}{k}\Big(\frac{c_{1}c_{5}}{2\gamma(K_{\star})}\Big)^{k}+c_{5}c_{1}^{k},
		\end{split}
	\end{equation}
	which entails that
	\begin{equation}
		\begin{split}
			\int_{0}^{L}u^{k}dx\leq\int_{0}^{L}u_{0}^{k}(x)+\frac{2\gamma(K_{\star})}{k}\Big(\frac{c_{1}c_{5}}{2\gamma(K_{\star})}\Big)^{k}+c_{5}c_{1}^{k}.
		\end{split} 
	\end{equation}
	We complete the proof.
\end{proof}
Next we will obtain the asymptotic behavior of solutions. We use the standard parabolic property to infer the regularity of $v$ and $u$ as follows.
\begin{lm} Let $(u,v)$ be the nonnegative global classical solution of ({\ref{Model 3}}) . Then there exist $\iota\in(0,1)$ and $C>0$ such that 
	\begin{equation}
		\begin{split}
			||u(\cdot,t)||_{C^{\iota,\frac{\iota}{2}}(\overline{[0,L]}\times[t,t+1])}\leq C
		\end{split}
	\end{equation}
	for all $t\geq1$ and 
	\begin{equation}
		\begin{split}
			||v(\cdot,t)||_{C^{2+\iota,1+\frac{\iota}{2}}(\overline{[0,L]}\times[t,t+1])}\leq C
		\end{split}
	\end{equation}
	for all $t\geq1$.
\end{lm}
{\bf{Proof of Theorem 1.1.}} Using Lemma 3.2 and parabolic regularity theorem, we can obtain the proof of Theorem 1.1.
\section{Large time behavior}
\subsection{Convergence of solutions.} Inspired by some ideas from \cite{H. Y. J. Y. K. Z. W}, we shall construct a Lyapounv functional to obtain the convergence of solutions. Since the linear function $r-\mu u$ is weaker than $\mu u(1-u)$ in \cite{H. Y. J. Y. K. Z. W}, then we construct the Lyapunov function $\mathcal{F}(t):=\int_{0}^{L}(u-\frac{r}{\mu})^{2}+\int_{0}^{L}(v-\frac{r}{\mu})^{2}$.

\begin{lm} Let $(u,v)$ be the solution of the system ({\ref{Model 3}}) with $r,\mu>0$. Then there exists a 
	\begin{equation}{\label{mu-def}}
		\begin{split}
			\mu>\frac{H_{0}}{4}
		\end{split}
	\end{equation}
	with $H_{0}:=\max_{0\leq v\leq\infty}\frac{|\gamma'(v)|^{2}}{\gamma(v)}$.
	then there exist two positive constants $\sigma,D>0$  such that the nonnegative function  
	\begin{equation}{\label{F-def}}
		\mathcal{F}(t):=\int_{0}^{L}(u-\frac{r}{\mu})^{2}+\int_{0}^{L}(v-\frac{r}{\mu})^{2}
	\end{equation}
	for all $t\geq0$ fulfills 
	\begin{equation}{\label{F-inequality}}
		\mathcal{F}'(t)\leq-D\mathcal{F}(t)\,\,\text{for all}\,\,t\in(0,T_{\max}).
	\end{equation}
\end{lm}
\begin{proof} We rewrite the $u-$ equation of the system ({\ref{Model 3}}) as 
	\begin{equation}{\label{u-e}}
		\begin{split}
			(u-\frac{r}{\mu})_{t}=\big(\gamma(v)u_{x}+\gamma'(v)uv_{x}\big)_{x}+r-\mu u.
		\end{split}   
	\end{equation}
	Then we multiply ({\ref{u-2-def}}) by $u-\frac{r}{\mu}$ and integrate by parts to infer that 
	\begin{equation}{\label{u-2-def}}
		\begin{split}
			&\frac{d}{dt}\int_{0}^{L}\big(u-\frac{r}{\mu}\big)^{2}dx+\int_{0}^{L}\gamma(v)u_{x}^{2}dx\\
			&=-\int_{0}^{L}\gamma'(v)uu_{x}v_{x}dx+\mu\int_{0}^{L}(u-\frac{r}{\mu})^{2}\\
		\end{split}
	\end{equation}
	for all $t>0$. Multiplying the $v-$ equation of the system ({\ref{Model 3}}) and integrating by parts, we deduce 
	\begin{equation}{\label{v-r-mu-def}}
		\begin{split}
			\frac{1}{2}\frac{d}{dt}\int_{0}^{L}(v-\frac{r}{\mu})^{2}&=\int_{0}^{L}(v-\frac{r}{\mu})\big[v_{xx}+(u-\frac{r}{\mu})-(v-\frac{r}{\mu})\big]dx\\
			&=-\int_{0}^{L}v_{x}^{2}dx-\int_{0}^{L}(v-\frac{r}{\mu           })^{2}dx+\int_{0}^{L}(v-\frac{r}{\mu})(u-\frac{r}{\mu})dx.
		\end{split}
	\end{equation}
	Multiplying ({\ref{v-r-mu-def}}) by a constant $\sigma>0$ and adding the result to ({\ref{u-2-def}}), we infer
	\begin{equation}
		\begin{split}
			&\frac{1}{2}\frac{d}{dt}\int_{0}^{L}\Big[(u-\frac{r}{\mu})^{2}+\sigma(v-\frac{r}{\mu})^{2}\Big]dx\\&=-\int_{0}^{L}\gamma(v)u_{x}^{2}dx-\sigma\int_{0}^{L}v_{x}^{2}dx+\int_{0}^{L}\gamma'(v)u_{x}v_{x}dx\\
			&\quad-\mu\int_{0}^{L}(u-\frac{r}{\mu})^{2}dx-\sigma\int_{0}^{L}(v-\frac{r}{\mu})^{2}dx\\&\quad+\sigma\int_{0}^{L}(u-\frac{r}{\mu})(v-\frac{r}{\mu})dx\\
			&=-\int_{0}^{L}\big(\gamma(v)u_{x}^{2}+\sigma v_{x}^{2}+\gamma'(v)u_{x}v_{x}\big)dx\\
			&\quad+\int_{0}^{L}\big[-\mu(u-\frac{r}{\mu})^{2}-\sigma(v-\frac{r}{\mu})^{2}+\sigma(u-\frac{r}{\mu})(v-\frac{r}{\mu})\big]dx\\
			&:=I_{1}+I_{2}.
		\end{split}
	\end{equation}
	For $I_{1}$, we can rewrite it as 
	\begin{equation}
		\begin{split}
			I_{1}=-X^{T}AX,\,\,X=\begin{bmatrix} u_{x}\\ v_{x }\end{bmatrix},\,\,A=\begin{bmatrix}
				\gamma(v) &\frac{\gamma'(v)}{2}\\ \frac{\gamma'(v)}{2}& \sigma \end{bmatrix},
		\end{split}
	\end{equation}
	where $X^{T}$ denotes the transpose of $X$. Hence $A$ is nonnegative definite and, $I_{1}\leq0$ if and only if
	\begin{equation}{\label{sigma-cond}}
		\begin{split}
			\sigma\geq\max_{0\leq v\leq\infty}\frac{|\gamma'(v)|^{2}}{4\gamma(v)}=\frac{H_{0}}{4}.
		\end{split}
	\end{equation}
	Similarly, we can rearrange $I_{2}$ as
	\begin{equation}
		\begin{split}
			I_{2}=-Y^{T}BY,\,\,Y=\begin{bmatrix}
				u-\frac{r}{\mu}\\v-\frac{r}{\mu}
			\end{bmatrix},\,\,B=\begin{bmatrix}
				\mu&-\frac{\sigma}{2}\\-\frac{\sigma}{2}&\sigma\end{bmatrix}.
		\end{split}
	\end{equation}
	Hence $I_{2}$ is positive definite if and only if
	\begin{equation}
		\mu>\frac{\sigma}{4}.
	\end{equation}	
	Therefore under the assumption ({\ref{mu-def}}), we can find a positive constant $\sigma$ fulfilling ({\ref{sigma-cond}}) and thus (\ref{F-inequality}) hold. Since $A$ is nonnegative definite and $B$ is positive definite, using  the definition of $\mathcal{F}(t)$ defined in ({\ref{F-def}}), we find a constant $D>0$  such that
	({\ref{F-inequality}}) holds.
\end{proof}
\begin{lm} Let $(u,v)$ be the solution of system ({\ref{Model 3}}) and $\mu>\frac{H_{0}}{16}$. Then 
	\begin{equation}
		||u(\cdot,t)-\frac{r}{\mu}||_{L^{\infty}}+||v(\cdot,t)-\frac{r}{\mu}||_{L^{\infty}}\leq Ce^{-\lambda t}
	\end{equation}
	for all $t\geq0$.
\end{lm}
\begin{proof} By Lemma 3.1, we infer
	\begin{equation}
		\mathcal{F}'(t)\leq-D\mathcal{F}(t)\,\,\text{for all}\,\,t\geq0,
	\end{equation}
	which entails that 
	\begin{equation}
		\int_{0}^{L}(u-\frac{r}{\mu})^{2}dx+\int_{0}^{L}(v-\frac{r}{\mu})^{2}dx\leq Ce^{-\lambda t}
	\end{equation}
	for all $t\geq0$. Using Lemma 2.6, we infer
	\begin{equation}
		||u-\frac{r}{\mu}||_{W^{1,\infty}}+||v(\cdot,t)-\frac{r}{\mu}||_{W^{1,\infty}}\leq c_{3}  
	\end{equation}
	for all $t>1$. Next we apply Gagliardo-Nirenberg inequality to get
	\begin{equation}
		||u(\cdot.t)-\frac{r}{\mu}||_{L^{\infty}}\leq c_{4}||u-\frac{r}{\mu}||_{W^{1,\infty}}^{\frac{1}{2}}||u-\frac{r}{\mu}||_{L^{2}}^{\frac{1}{2}}\leq c_{1}^{\frac{1}{2}}c_{3}^{\frac{1}{2}}c_{4}e^{-\frac{c_{2}}{2}t}
	\end{equation}
	for all $t>t_{0}$. Similarly, we infer
	\begin{equation}
		||u-\frac{r}{\mu}||_{L^{\infty}}+||v-\frac{r}{\mu}||_{L^{\infty}}\leq c_{6}e^{-c_{7}t}
	\end{equation}
	for all $t>t_{0}$. We finish the proof.
\end{proof}
{\bf{Proof of Theorem 1.2.}} By Lemmas 3.2, we can finish the proof of Theorem 1.2.

\section{Simulation and Disscusion}\label{NUm}
In this section, we present numerical simulations of the system \ref{Model 3} in one, two, and three dimensions using a finite difference scheme. We fix the motility function $\gamma(v)$ as 
$$
\gamma(v) = \frac{1}{1+e^{8(v-1)}},
$$
whose derivative $\gamma'(v) = -8e^{8(v-1)}/(1 + e^{8(v-1)})^2$ ensures that $\gamma(v)$ satisfies the required assumptions. According to Theorem , globally bounded solutions exist when $\mu > 0$ and $\gamma > 0$, which is confirmed by our numerical results. To further explore the solution behavior, we also examine the case $\mu = 0$ across different dimensions. For the one-dimensional case, we first consider $\mu = 0$ with initial data given as a small random perturbation around $(1000, 1000)$. As shown in Fig.~\ref{fig_Num_1}, the solution $u(x,t)$ grows with time, eventually blowing up and becoming unbounded, while $v(x,t)$ exhibits similar behavior (omitted for brevity). In contrast, when $\mu > 0$, the solutions remain bounded regardless of the initial conditions. Figs.~\ref{fig_Num_2} and \ref{fig_Num_3} illustrate this for different initial values: with small perturbations around $(\gamma/\mu, \gamma/\mu)$ ($\mu = 0.01$, $\gamma = 1$) for $(u_0(x), v_0(x))$, the solutions initially fluctuate but eventually stabilize to $(\gamma/\mu, \gamma/\mu)$, consistent with Theorem . The same stabilization occurs even for large initial data $(10^4, 10^4)$, reinforcing the theoretical prediction. In summary, while the solution blows up for $\mu = 0$, it converges to the steady state $(\gamma/\mu, \gamma/\mu)$ when $\mu > 0$, demonstrating the critical role of $\mu$ in determining long-term behavior.   
\begin{figure}[!ht]
	\scriptsize
	\begin{center}
		\begin{tabular}{c}
			\includegraphics[width=0.7\textwidth]{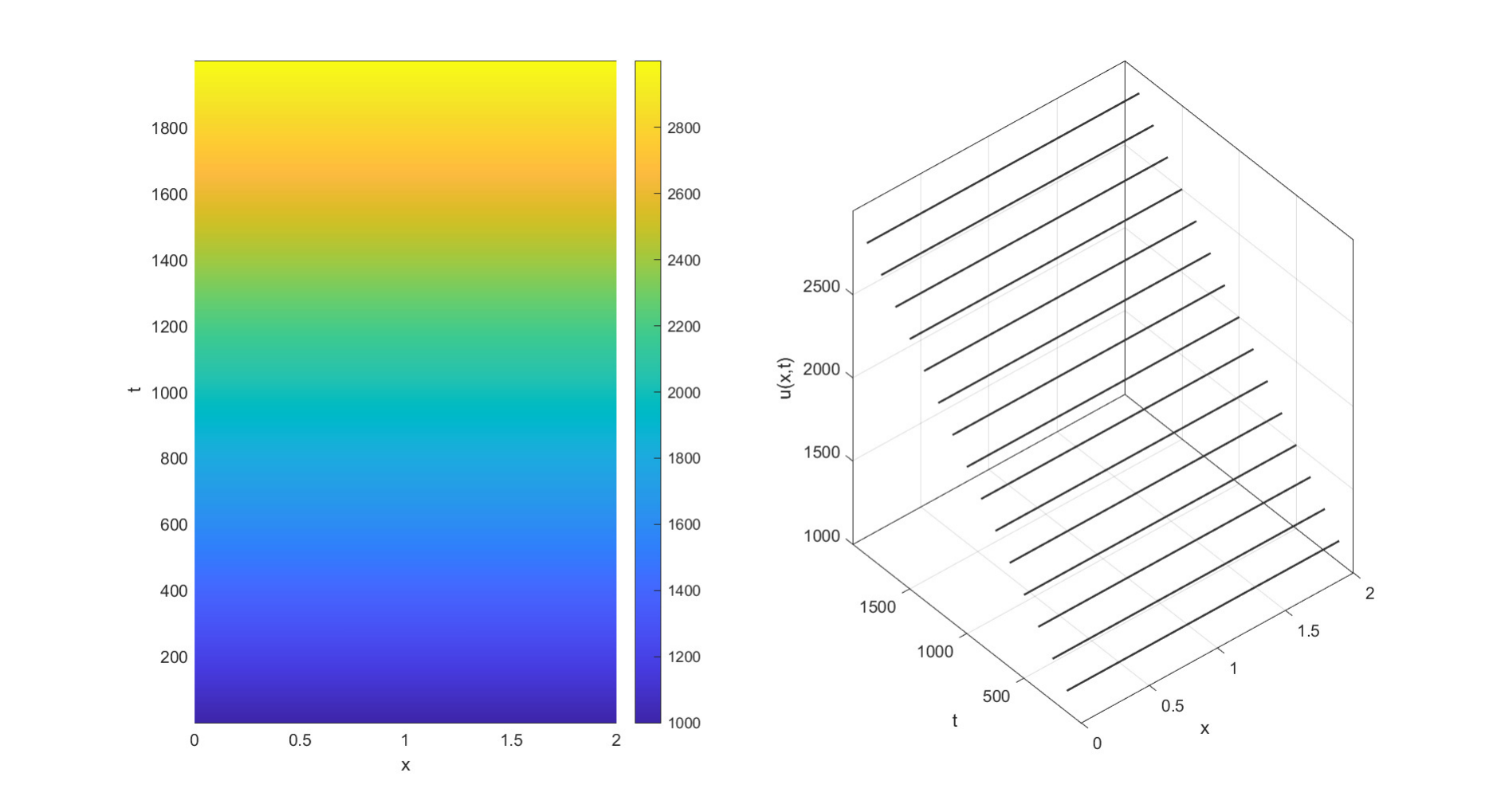}
		\end{tabular}
		\vspace{-0.1cm}
		\caption{\footnotesize Numerical simulation for $u(x,t)$ of the system \ref{Model 3} in one dimensional case in an interval $[0,2]$, where the initial data $(u_0, v_0)$ are chosen as a small random perturbation of $(1000, 1000)$ with $\mu=0$. }\label{fig_Num_1}
	\end{center}
	\vspace{0.2cm}
\end{figure}

\begin{figure}[!ht]
	\scriptsize
	\begin{center}
		\begin{tabular}{c}
			\includegraphics[width=0.9\textwidth]{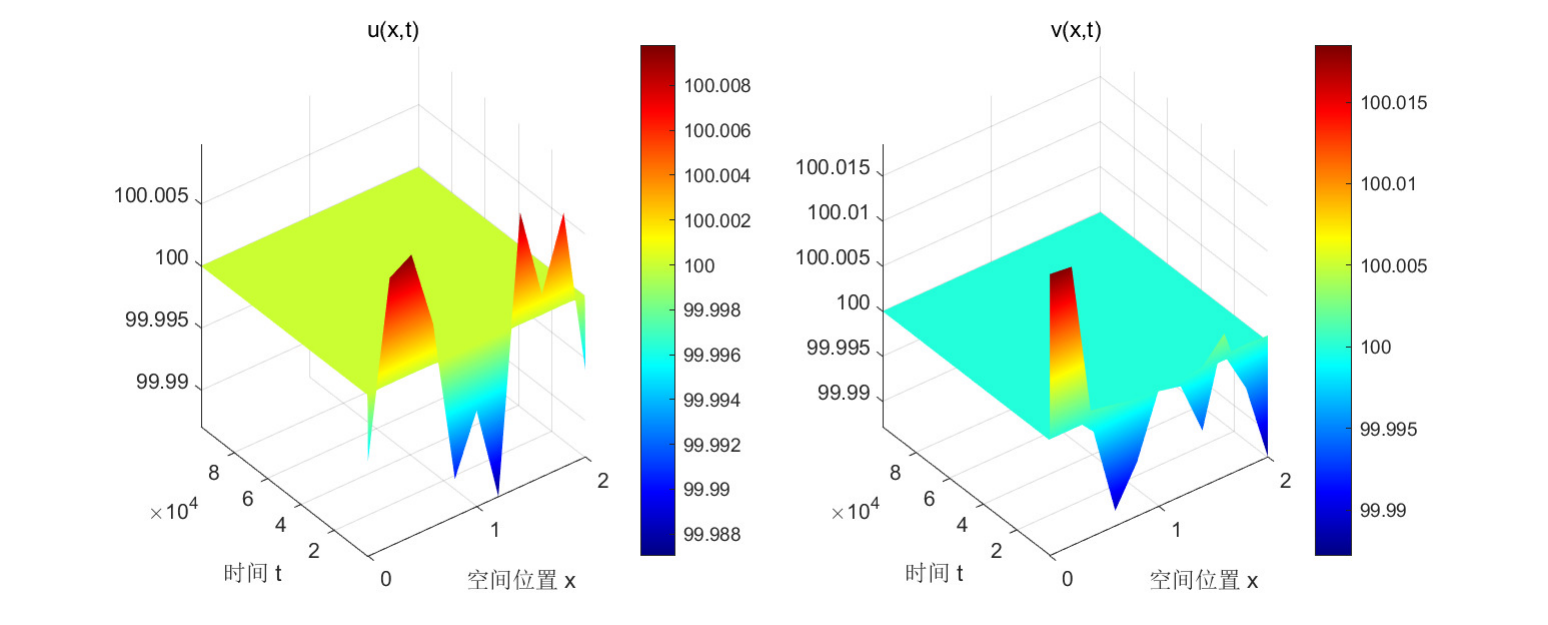}
		\end{tabular}
		\vspace{-0.1cm}
		\caption{\footnotesize Numerical simulation for $u(x,t)$ and $v(x,t)$ of the \ref{Model 3} in one dimensional case in an interval $[0,2]$, where the initial data $(u_0, v_0)$ are chosen as a small random perturbation of $(\gamma/\mu, \gamma/\mu)$ with $\mu=0.01$ and $\gamma = 1$. }\label{fig_Num_2}
	\end{center}
	\vspace{0.2cm}
\end{figure}

\begin{figure}[!ht]
	\scriptsize
	\begin{center}
		\begin{tabular}{c}
			\includegraphics[width=0.9\textwidth]{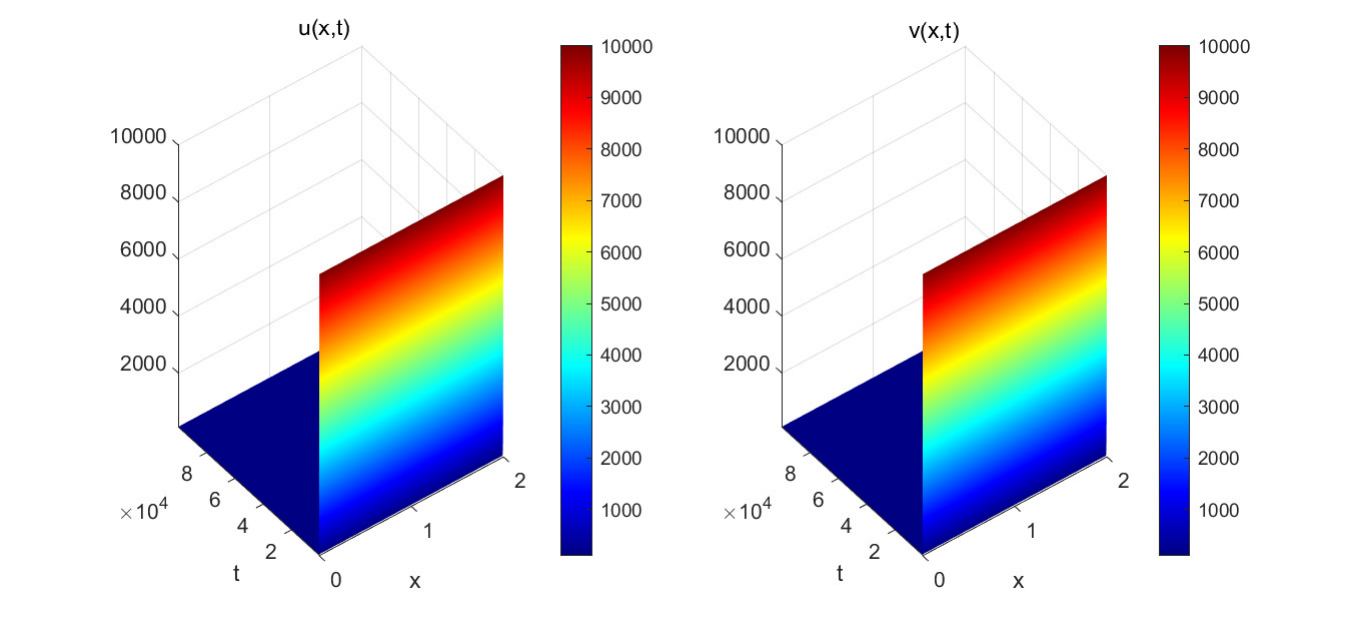}
		\end{tabular}
		\vspace{-0.1cm}
		\caption{\footnotesize Numerical simulation for $u(x,t)$ and $v(x,t)$ of the \eqref{Model 3} in one dimensional case in an interval $[0,2]$, where the initial data $(u_0, v_0)$ are chosen as a small random perturbation of $(10^4, 10^4)$ with $\mu=0.01$ and $\gamma = 1$. }\label{fig_Num_3}
	\end{center}
	\vspace{0.2cm}
\end{figure}

The above results concern the one-dimensional case. Next, we consider the two-dimensional case, which yields the same conclusions as in one dimension. We present these results in Figs.~\ref{fig_Num_4}-\ref{fig_Num_6}. First, for $\mu = 0$ with initial data consisting of a small random perturbation around $(1000, 1000)$, Fig.~\ref{fig_Num_4} shows that both $u(x,t)$ and $v(x,t)$ grow with time, eventually blowing up and becoming unbounded. In contrast, when $\mu > 0$, the solutions remain bounded regardless of initial conditions.
Figs.~\ref{fig_Num_5} and \ref{fig_Num_6} demonstrate this behavior for different initial values. With small perturbations around $(\gamma/\mu, \gamma/\mu)$ (where $\mu = 0.001$ and $\gamma = 1$) for $(u_0(x), v_0(x))$, the solutions exhibit pattern formation during the temporal merging process before eventually stabilizing to $(\gamma/\mu, \gamma/\mu)$, consistent with Theorem . This stabilization occurs even for large initial data $(10^4, 10^4)$, further confirming the theoretical prediction.
Notably, when the initial data equals the steady state value $(\gamma/\mu, \gamma/\mu)$, $u(x,t)$ shows an interesting pattern formation process. As shown in Fig.~\ref{fig_Num_5}, cells are initially uniformly distributed with density $u_0$. After a short time, this uniform steady state begins to break down and form a honeycomb structure. As time progresses, these small aggregations merge to form larger structures. However, this pattern formation does not emerge when the initial values differ from the steady state.
In short, while solutions blow up for $\mu = 0$, they converge to the steady state $(\gamma/\mu, \gamma/\mu)$ when $\mu > 0$, highlighting $\mu$'s critical role in determining long-term behavior in two dimendional case.  

\begin{figure}[!ht]
	\scriptsize
	\begin{center}
		\begin{tabular}{c}
			\includegraphics[width=0.9\textwidth]{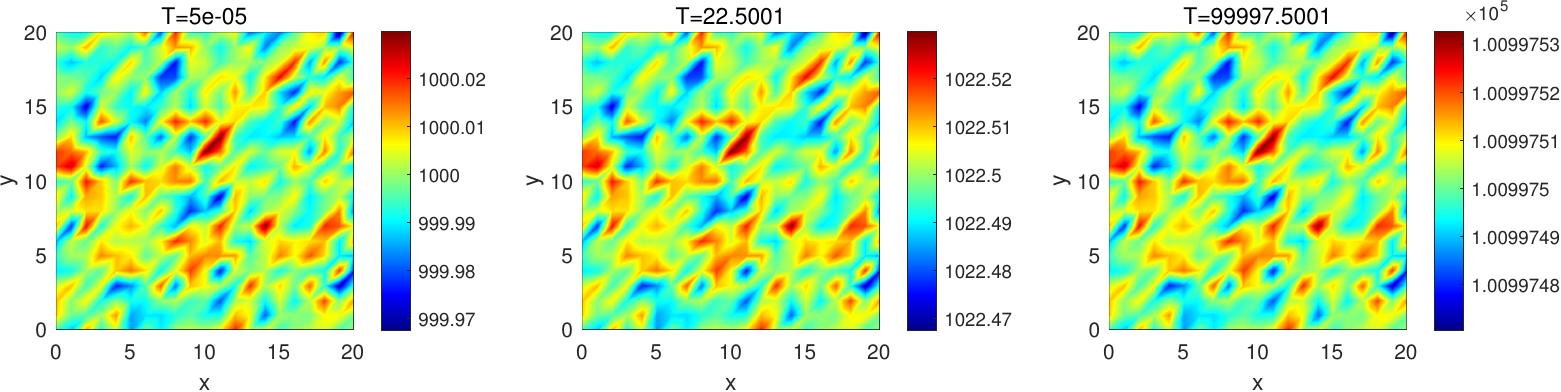}\\
			\includegraphics[width=0.9\textwidth]{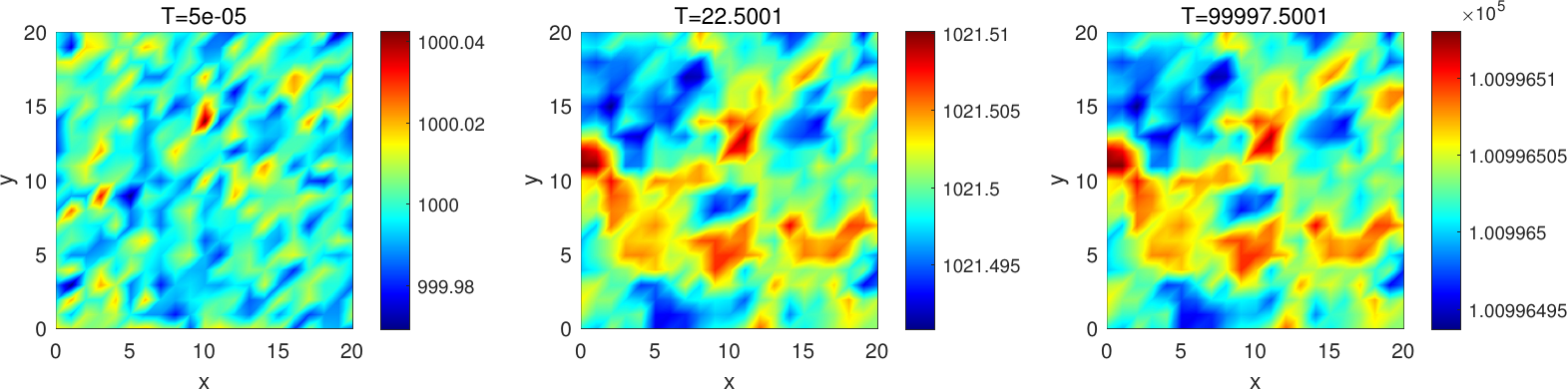}
		\end{tabular}
		\vspace{-0.1cm}
		\caption{\footnotesize  Numerical simulation for $u(x,t)$ and $v(x,t)$ of the \eqref{Model 3} in two dimensional case in an rectangle domain $[0,2]\times [0,2]$, where the initial data $(u_0, v_0)$ are chosen as a small random perturbation of $(1000, 1000)$ with $\mu=0$.}\label{fig_Num_4}
	\end{center}
	\vspace{0.2cm}
\end{figure}

\begin{figure}[!ht]
	\scriptsize
	\begin{center}
		\begin{tabular}{c}
			\includegraphics[width=0.9\textwidth]{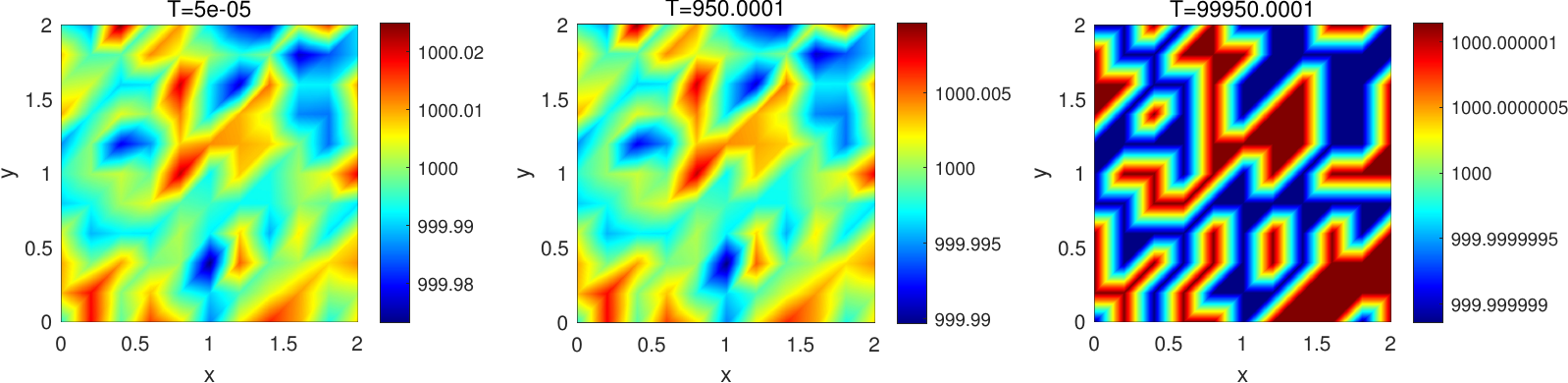}\\
			\includegraphics[width=0.9\textwidth]{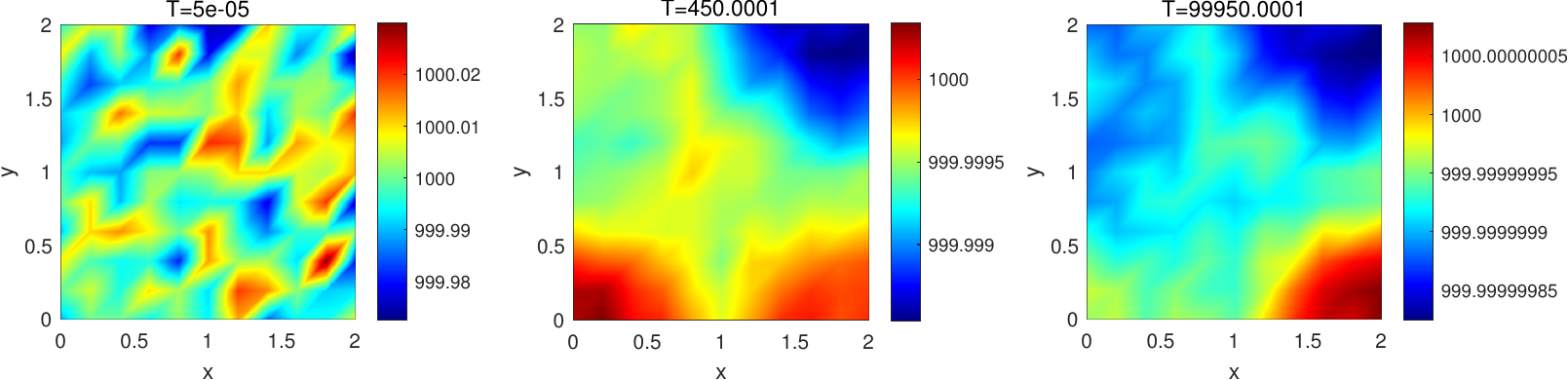}
		\end{tabular}
		\vspace{-0.1cm}
		\caption{\footnotesize  Numerical simulation for $u(x,t)$ and $v(x,t)$ of the \eqref{Model 3} in two dimensional case in an interval $[0,2]\times [0,2]$, where the initial data $(u_0, v_0)$ are chosen as a small random perturbation of $(\gamma/\mu, \gamma/\mu)$ with $\mu=0.001$ and $\gamma = 1$.}\label{fig_Num_5}
	\end{center}
	\vspace{0.2cm}
\end{figure}

\begin{figure}[!ht]
	\scriptsize
	\begin{center}
		\begin{tabular}{c}
			\includegraphics[width=0.9\textwidth]{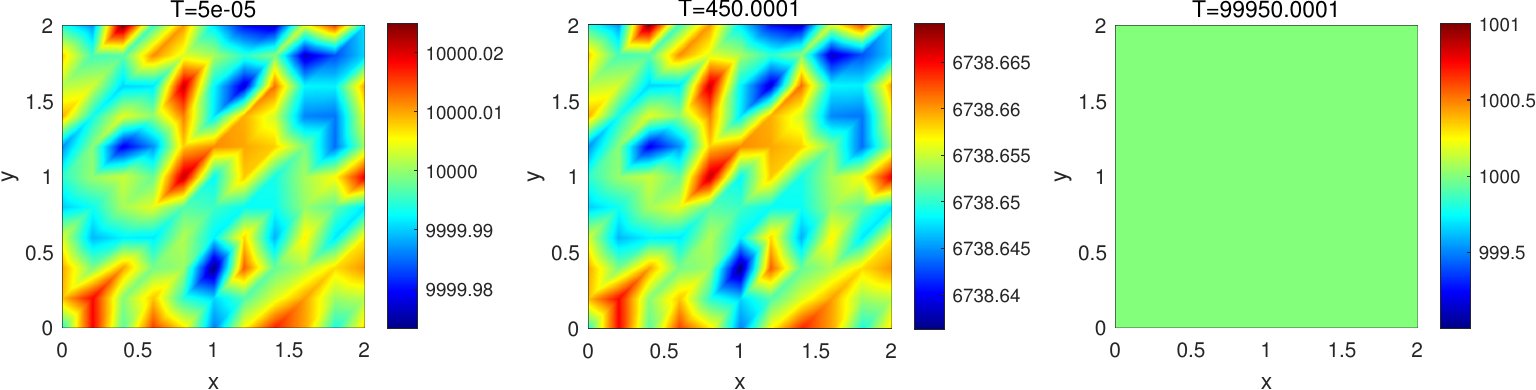}\\
			\includegraphics[width=0.9\textwidth]{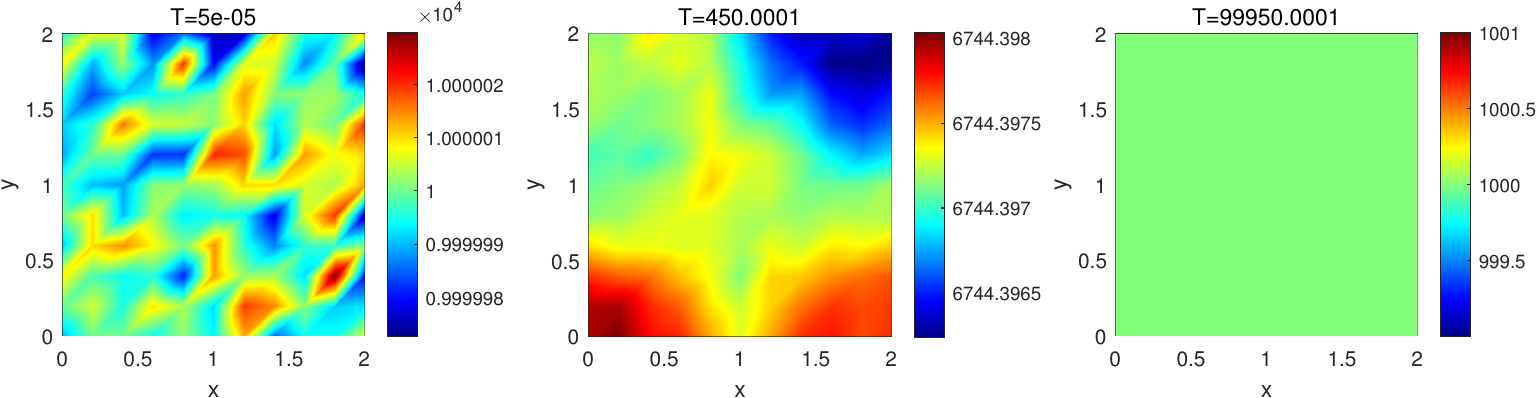}
		\end{tabular}
		\vspace{-0.1cm}
		\caption{\footnotesize  Numerical simulation for $u(x,t)$ and $v(x,t)$ of the \eqref{Model 1} in two dimensional case in an interval $[0,2]\times [0,2]$, where the initial data $(u_0, v_0)$ are chosen as a small random perturbation of $(10^4, 10^4)$ with $\mu=0.001$ and $\gamma = 1$.}\label{fig_Num_6}
	\end{center}
	\vspace{0.2cm}
\end{figure}

The above results concern the two-dimensional case. Next, we consider the three-dimensional case, which yields the same conclusions as in one and two dimensions. We present these results in Figs.~\ref{fig_Num_7}-\ref{fig_Num_9}. First, for $\mu = 0$ with initial data consisting of a small random perturbation around $(1000, 1000)$, Fig.~\ref{fig_Num_7} shows that both the results of $u(x,t)$ and $v(x,t)$ at time $t = 99990.0001$, the value is $10^5$, which is very large. Actually, the solution $u(x,t)$ grows with time, eventually blowing up and becoming unbounded. In contrast, when $\mu > 0$, the solutions remain bounded regardless of initial conditions.
Figs.~\ref{fig_Num_8} and \ref{fig_Num_9} demonstrate this behavior for different initial values. With small perturbations around $(\gamma/\mu, \gamma/\mu)$ (where $\mu = 0.001$ and $\gamma = 1$) for $(u_0(x), v_0(x))$, the solutions exhibit pattern formation during the temporal merging process before eventually stabilizing to $(\gamma/\mu, \gamma/\mu)$, consistent with Theorem . 
Meanwhile, when the initial data equals the steady state value $(\gamma/\mu, \gamma/\mu)$, $u(x,t)$ shows an interesting pattern formation process as two dimensional case, which is shown in Fig.~\ref{fig_Num_8}. When the initial data is chosen as the $(10^4, 10^4)$, the value of $u$ and $v$ are all bounded. But $u$ tends to $\gamma\/mu$, $v$ tends to small pertubation of the initial data.
In short, while solutions blow up for $\mu = 0$, they become bounded when $\mu > 0$, highlighting $\mu$'s critical role in determining long-term behavior. Notabaly, when $f(u) = ru - \mu u$, $D(u,v) =1$, the solution will be blow up in three dimensional case when $\mu >0$. However, in our case,  $f(u) = ru - \mu u$, $D(u,v) =r(v)$, the solution is not blow up, which demonstrates the importance of the density-suppressed term.

\begin{figure}[!ht]
	\scriptsize
	\begin{center}
		\begin{tabular}{c}
			\includegraphics[width=0.9\textwidth]{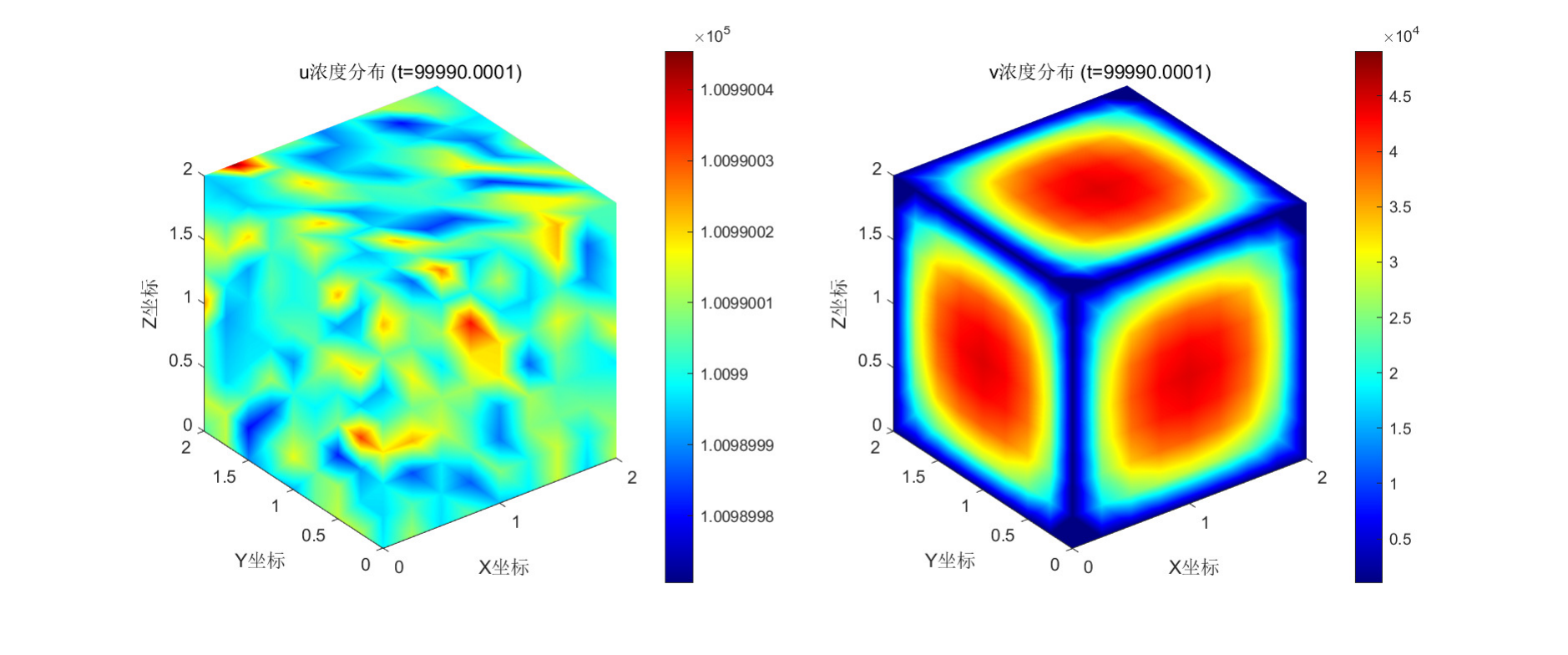}
		\end{tabular}
		\vspace{-0.1cm}
		\caption{\footnotesize Numerical simulation for $u(x,t)$ of the \eqref{Model 3} in three dimensional case in a cube $[0,2]\times [0,2] \times [0,2]$, where the initial data $(u_0, v_0)$ are chosen as a small random perturbation of $(1000, 1000)$ with $\mu=0$. }\label{fig_Num_7}
	\end{center}
	\vspace{0.2cm}
\end{figure}

\begin{figure}[!ht]
	\scriptsize
	\begin{center}
		\begin{tabular}{c}
			\includegraphics[width=0.9\textwidth]{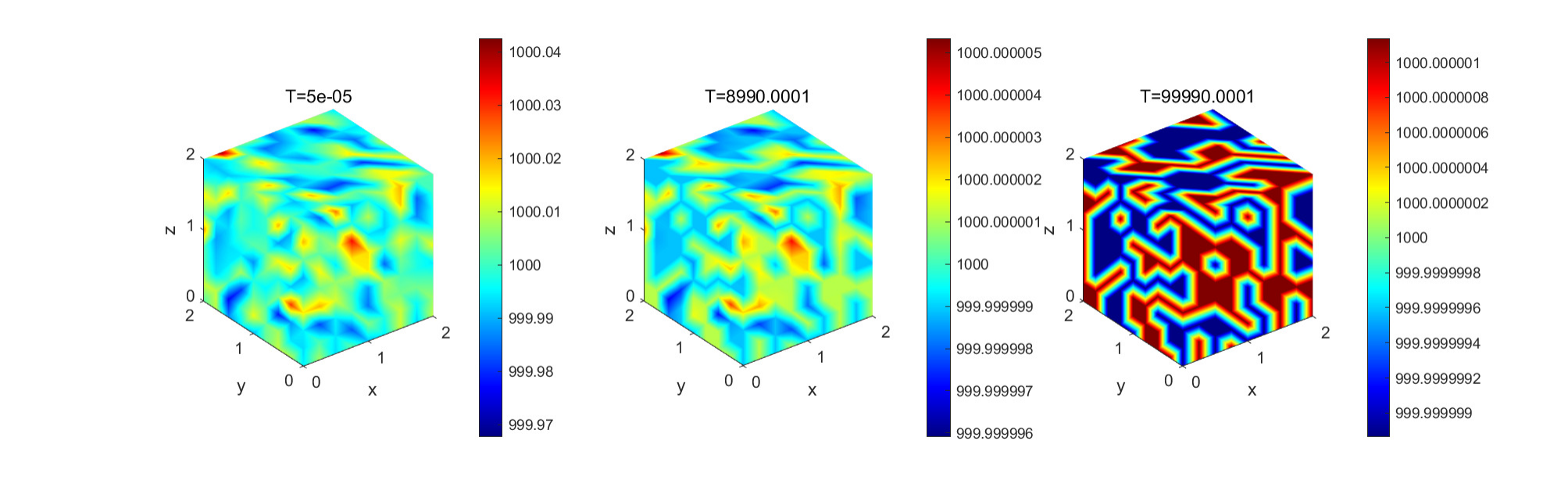}\\
			\includegraphics[width=0.9\textwidth]{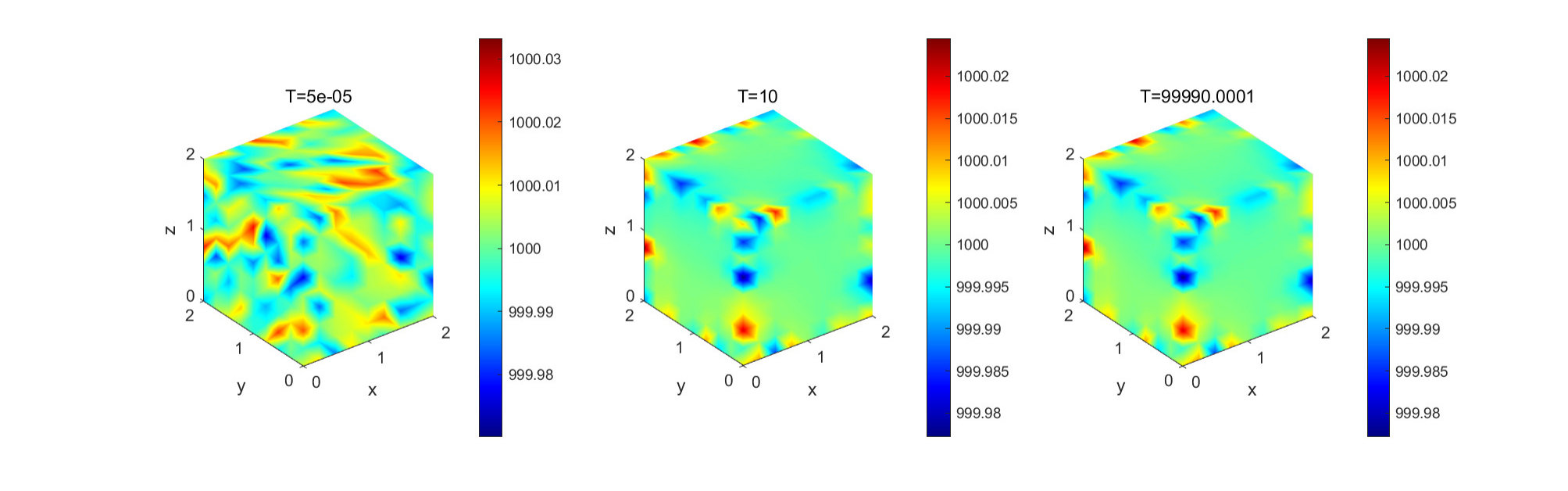}
		\end{tabular}
		\vspace{-0.1cm}
		\caption{\footnotesize Numerical simulation for $u(x,t)$ and $v(x,t)$ of the \eqref{Model 3} in three dimensional case in a cube $[0,2]\times [0,2] \times [0,2]$, where the initial data $(u_0, v_0)$ are chosen as a small random perturbation of $(\gamma/\mu, \gamma/\mu)$ with $\mu=0.001$ and $\gamma  =1$. }\label{fig_Num_8}
	\end{center}
	\vspace{0.2cm}
\end{figure} 

\begin{figure}[!ht]
	\scriptsize
	\begin{center}
		\begin{tabular}{c}
			\includegraphics[width=0.9\textwidth]{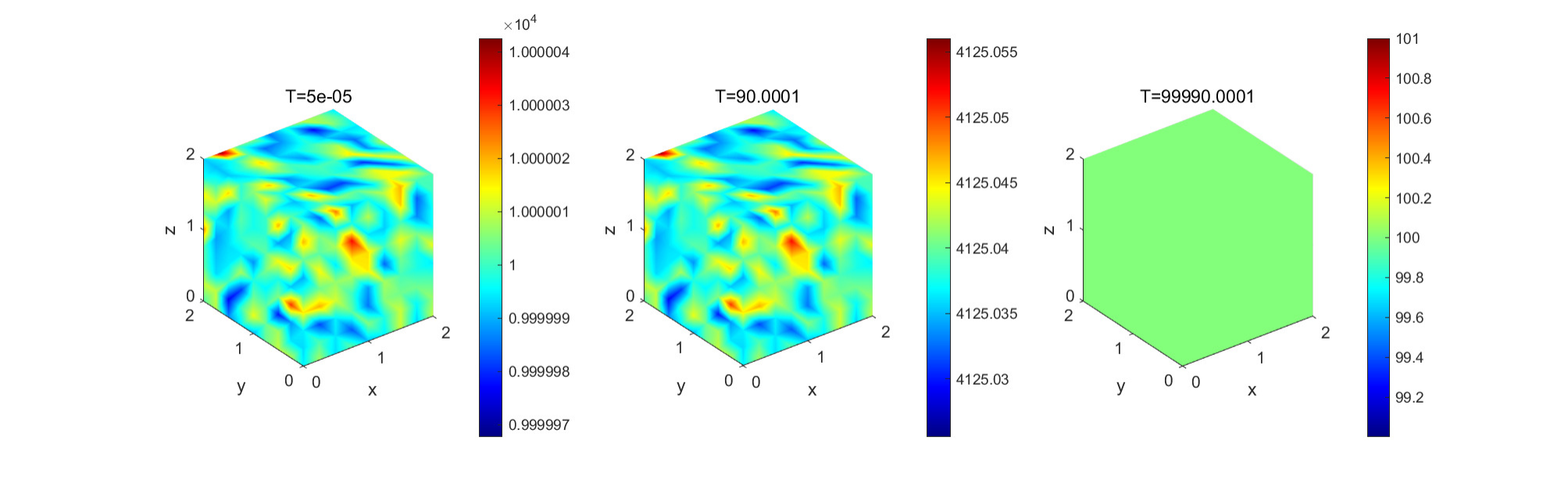}\\
			\includegraphics[width=0.9\textwidth]{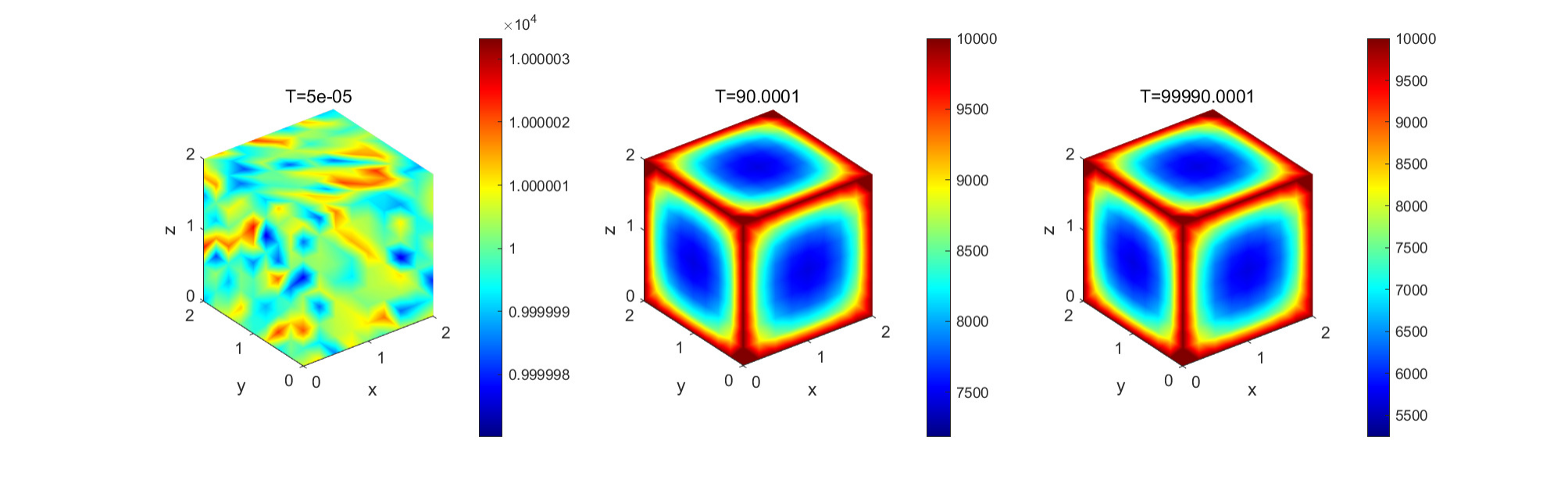}
		\end{tabular}
		\vspace{-0.1cm}
		\caption{\footnotesize Numerical simulation for $u(x,t)$ and $v(x,t)$ of the \eqref{Model 3} in three dimensional case in a cube $[0,2]\times [0,2] \times [0,2]$, where the initial data $(u_0, v_0)$ are chosen as a small random perturbation of $(10^4, 10^4)$ with $\mu=0.01$ and $\gamma  =1$. }\label{fig_Num_9}
	\end{center}
	\vspace{0.2cm}
\end{figure}

In summary, the solution of Model~\ref{Model 3} exhibits blow-up when $\mu = 0$. However, for $\mu > 0$, the solution remains bounded for all initial data. 
Our work focuses on the system with $f(u) = \gamma - \mu u$, but it is worth noting that other cases have been studied in the literature. For instance, when $D(u, v) = 1$ and $\chi(u, v) = \chi > 0$ (where $\chi$ is a constant), the system (1.2) reduces to the minimal Keller-Segel model. If $f(u) = 0$, the solution does not blow up in one-dimensional settings, whereas in higher dimensions (two or more), blow-up behavior whether in finite or infinite time depends on the domain and initial conditions. Another case arises when $f(u) = \mu u(1 - u)$ with a motility function $\gamma(v) \in C^3([0, \infty))$ satisfying $\gamma(v) > 0$, $\gamma'(v) < 0$ for all $v \geq 0$, $\lim_{v \to \infty} \gamma(v) = 0$, and $\lim_{v \to \infty} \frac{\gamma'(v)}{\gamma(v)}$ existing; here, solutions are globally bounded in two-dimensional spaces. 
In the latter case, the inequality
\[
\int_0^{t+\tau} \int_\Omega |\Delta v|^2 \, dx \, dt \leq C
\]
plays a crucial role. However, in our setting with $f(u) = \gamma - \mu u$, this key inequality does not hold, preventing us from deriving boundedness results theoretically. 
Nevertheless, Our numerical simulations reveal a fundamental dichotomy in solution behavior governed by the parameter $\mu$. For $\mu > 0$, we observe globally bounded solutions across all considered dimensions (1D, 2D, and 3D) regardless of initial conditions, suggesting that positive $\mu$ values exert a stabilizing effect on the system dynamics. In stark contrast, the $\mu = 0$ case consistently exhibits finite-time blow-up in all computational scenarios. While these numerical findings strongly indicate the stabilizing role of $\mu > 0$, the theoretical underpinnings remain incomplete due to the breakdown of essential analytical inequalities in our framework. This theoretical gap motivates the development of innovative analytical techniques to overcome current limitations, representing a crucial direction for future research.

\end{document}